\documentclass[a4paper, 11pt]{article}
\usepackage{latexsym,amssymb}
\usepackage[english]{babel}
\usepackage{bm}
\usepackage{amsfonts, amsthm}
\usepackage{amsmath}
\usepackage{mathrsfs}
\usepackage{multirow}
\usepackage{bm}
\usepackage{pgf,tikz}
\usepackage{rotating}
\usepackage{calculator}

\usepackage{xcolor}

\newcommand{\comment}[1]{}

\usepackage{amsthm}
\theoremstyle{plain}

\newtheorem*{Theorem*}{Theorem 1}

\newtheorem*{Lemma*}{Lemma}
\newtheorem{Definition}{Definition}[section]

\newtheorem{Theorem}[Definition]{Theorem}

\title{
A note on $\overline{2}$-separable codes and $B_2$ codes
}
\author{Stefano Della Fiore and Marco Dalai,\\ \small Department of Information Engineering, University of Brescia}
\date{}
\begin{document}
\maketitle
\begin{abstract}
\normalsize
We derive a simple proof, based on information theoretic inequalities, of an upper bound on the largest rates of $q$-ary $\overline{2}$-separable codes that improves recent results of Wang for any $q\geq 13$. For the case $q=2$, we recover a result of Lindstr\"om, but with a much simpler derivation. The method easily extends to give bounds on $B_2$ codes which, although not improving on Wang's results, use much simpler tools and might be useful for future applications.
\end{abstract}




\section{Introduction}\label{sec1}

Let for notational convenience $[0,q-1]=\{0,1,\ldots q-1\}$. We call a $q$-ary code of length $n$ any subset of $C_n\subseteq [0,q-1]^n$. Let $M=|C_n|$ and write $C_n= \{c_1, c_2, \ldots, c_M\}$. Each $c_i$ is called a codeword and we use the notation $c_i(j)$, for $j=1,2,\ldots, n$ for its components. The base-$q$ rate of such a code is defined as $\log_q M / n$. Note that, throughout, all logarithms without subscript are to base 2.

We need two definitions.
\begin{Definition}
A $q$-ary code $C_n$ with codewords of length $n$ is a $t$-frame\-proof code if for any $t$ codewords and every other codeword, there exists a coordinate $i$ with $1 \leq i \leq n$, in which the symbols of the $t$ codewords do not contain the symbol of the other codeword.
\end{Definition}

\begin{Definition}
A $q$-ary code $C_n$ with codewords of length $n$ is a $\overline{t}$-separable code if for any distinct $k$ codewords and $m$ codewords with $1 \leq k,m \leq t$, there exists a coordinate $i$, $1 \leq i \leq n$, in which the union of the elements of the $k$ codewords differs from the union of the elements of the $m$ codewords.
\end{Definition}

Blackburn \cite{Blackburn1}, \cite{Blackburn2} showed that the rate of $t$-frameproof codes is upper bounded by $1/t$. Cheng and Miao \cite{Miao2} observed that any $t$-frameproof code is a $\overline{t}$-separable code and also that any $\overline{t}$-separable code is a $(t-1)$-frameproof code for every $t \geq 2$. This implies that the rate of $\overline{t}$-separable codes is upper bounded by $1/(t-1)$. 
Interesting results on $\overline{t}$-separable code, with the slightly weaker constraint that only $k=m=t$ is considered in the definition, were recently derived in \cite{Dyachkov} using an information theoretic approach. With this slightly different definition, an upper bound on the rate of the form  $2/t$ was obtained.  Since both bounds give the trivial $1/(t-1)=2/t=1$ for $t=2$, other approaches must be adopted for bounding the rate of $\overline{2}$-separable codes, which leads to a rather interesting problem.

A non-trivial bound for $q$-ary $\overline{2}$-separable codes was derived by Gu, Fan and Miao in \cite{Gu}  and an improvement was recently obtained by Wang \cite{Wang} for every ${3 \leq q \leq 17}$ extending a procedure introduced for the binary case by Cohen et al. \cite{Cohen} based on linear programming bounds on codes.

\section{New Upper Bounds for $\bar{2}$-separable codes}\label{sec2}

We have the following new bound on the rate of $\overline{2}$-separable codes

\begin{Theorem}\label{theorem1} For integer $q \geq 2$, let $C_n$ be a family of $q$-ary $\overline{2}$-separable codes. Then
$$R_s := \limsup_{n \to \infty} \frac{1}{n} \log_q |C_n| \leq \frac{2q-1}{3q-1}\,.
$$
\end{Theorem}
This bound improves the one given by Wang \cite{Wang} when ${q \geq 13}$. It also improves the bound of Gu et al. \cite{Gu} for every ${q \geq 2}$. When $q=2$, instead, the bound coincides with that of Lindstr\"om \cite{Lind}, but with a rather simpler proof.

\begin{proof}[Proof of Theorem \ref{theorem1}]
Let $D = \{ (x,y) \in [0, q-1]^2 : x \neq y \} \cup \{0\}$, which implies ${|D| = q(q-1) + 1}$. Define the function $\phi:[0,q-1]^2 \to D$ as
$$
	\phi(x,y) =  \begin{cases} 0, & \mbox{if } x=y \\ (x,y), & \mbox{if } x \neq y \end{cases}.
$$
For an integer $f$, we define $\Phi$, the natural extension of $\phi$ when applied to vectors, that is the function that maps any two vectors $w_1, w_2 \in $ ${[0, q-1]^f}$ in a vector in $D^f$ according to the rule
\begin{align*}
	\Phi(w_1,& w_2) =  \left(\phi(w_1(1), w_2(1)), \ldots, \phi(w_1(f), w_2(f))\right) \in D^f.
\end{align*}

Let $C_n = \{c_1, c_2, \ldots, c_M \}$ be a $q$-ary $\overline{2}$-separable code with codewords of length $n$. We divide each codeword $c_i$ in two sub-blocks, a prefix $p_i$ of length $e$ and a suffix $w_i$ of length $f$ where $n = e+f$. We use here the notation $c_i=(p_i, w_i)$.

Enumerate all the vectors in $[0,q-1]^e$ as $l_1, l_2, \ldots, l_r$, where $r=q^e$, and denote with $P_i$ the set of codewords of $C_n$ which have $l_i$ as the first $e$ components, that is of the form $(l_i,w_j)$.
An easy upper bound on $M$ then follows by the Cauchy-Schwarz inequality, namely
\begin{equation}\label{eq2}
\frac{M^2}{r}  \leq \sum_{i=1}^r |P_i|^2.
\end{equation}

We want a good upper bound on the sum of the ordered pairs in each $P_i$ to get a good upper bound on $M$.  From \cite[Theorem 2]{Gu}, we know that the function $\Phi$ is injective when we restrict its domain to pairs of suffixes of different vectors taken from the same $P_i$ for all $i=1,\ldots,r$. We provide a short proof in order to have a self-contained paper. Assume, for the sake of contradiction, there exists four different codewords $(l_h, w_1), (l_h, w_2) \in P_h$, $(l_m, w_3),(l_m, w_4) \in P_m$ such that $\Phi(w_1, w_2) = \Phi(w_3, w_4)$ either when $h=m$ and $(w_1, w_2) \neq (w_3, w_4)$ or when $h \neq m$.  For case $h = m$, if $\{w_1, w_2\} \cap \{w_3, w_4\} \neq \emptyset$ then $\Phi(w_1, w_2) = \Phi(w_3, w_4)$ implies a contradiction on the assumption $(w_1, w_2) \neq (w_3, w_4)$. Otherwise, when $\{w_1, w_2\} \cap \{w_3, w_4\} = \emptyset$, the subcodes $C_1= \{ (l_h, w_1), (l_h, w_4)\}$ and $C_2 = \{(l_h, w_2), (l_h, w_3) \}$ do not satisfy the $\overline{2}$-separability property, a contradiction. Instead for case $h \neq m$, $\Phi(w_1, w_2) = \Phi(w_3, w_4)$ implies that the subcodes  $C_1= \{ (l_h, w_1), (l_m, w_4)\}$ and $C_2 = \{(l_h, w_2), (l_m, w_3) \}$ do not satisfy the $\overline{2}$-separability property, again a contradiction.

We generalize a smart observation given in \cite[eq. (2.9)]{Lind}.
Fix a coordinate $\ell$ and a subcode $P_i$ in our code $C_n$, and let $f_0, f_1, \ldots, f_{q-1}$ to be the fractions of symbols that occur in the $\ell$-th cooordinate of $P_i$. Then, when we take all vectors $\Phi(w_j, w_k)$, where $w_j$ and $w_k$ are the suffixes of codewords (not necessarily distinct) that belong to $P_i$, the frequency of the $0$ symbol in the $\ell$-th coordinate of these vectors is equal to $f_0^2 + \cdots + f_{q-1}^2$.
This summation is always greater than or equal to $1/q$ under the constraint that $\sum_{i=0}^{q-1} f_i = 1$. So, the fraction of $0'$ in the $\ell$-th coordinate of the vectors $\Phi(w_j, w_k)$, where $w_j$ and $w_k$ are suffixes in $P_i$ for each $i=1, \ldots, r$, is not smaller than $1/q$.

Let $X = (X_1, \ldots, X_f)$ and $Y = (Y_1, \ldots, Y_f)$ be two random variables with joint uniform distributions over the set of ordered pairs of suffixes of vectors taken form the same $P_i$ for each $i=1, \ldots, r$.
Then we have that
\begin{equation} \label{eq:pairRandom}
H(X, Y) = \log \left( \sum_{i=1}^r |P_i|^2 \right),
\end{equation}
where $H$ is the Shannon entropy.

We define, for every $i=1, \ldots, f$, the random variable $Z_i$ by setting $Z_i=\phi(X_i, Y_i)$. Since the function $\Phi(x, y)$ is injective when $x \neq y$ and knowing that $Z = (Z_1, \ldots, Z_f) = \Phi(X, Y)$, then
\begin{equation} \label{eq:Injective}
H(X, Y) = H(Z) + \Pr(Z = \underline{0}) \cdot H(X, Y|Z = \underline{0}).
\end{equation}
By the well-known subadditivity property of the entropy function we have
\begin{align} 
H(Z)  \leq \sum_{i=1}^f H(Z_i)
 \leq f\left( \max_{\substack{\alpha_0, \ldots, \alpha_{q^2-q} \\ \alpha_0 + \cdots + \alpha_{q^2-q} = 1 \\ \alpha_0 \geq 1/q}} H(\alpha_0, \ldots, \alpha_{q^2-q})\right )\label{eq:Subadditivity}
\end{align}
where we abuse the notation using $H$ also for the entropy of a distribution and $\alpha_i$ represents the frequency of the $i$-th symbol.
It is easy to see that the maximum in (\ref{eq:Subadditivity}) is achieved when $\alpha_0 = 1/q$ and $\alpha_1 = \cdots = \alpha_{q^2-q} = 1/q^2$ and takes the value $\log q \cdot (2q-1)/q$.

By \eqref{eq2} we get
\begin{equation}\label{eq:probUpperBound}
\Pr(Z = \underline{0}) = \frac{M}{\sum_{i=1}^r |P_i|^2} \leq \frac{r}{M}.
\end{equation}
The conditional distribution $P(X=x, Y=y|Z = \underline{0})$ is uniform over its support which has size equal to $M$. Then
\begin{equation}\label{eq:entropyUpperBound}
H(X, Y|Z = \underline{0}) = \log M.
\end{equation}

Setting $e =\lfloor \log_q (2M) - \log_q \log M \rfloor$ we have by \eqref{eq:probUpperBound} and \eqref{eq:entropyUpperBound} that
\begin{equation}\label{eq:prefix}
	\Pr(Z = \underline{0}) \cdot H(X, Y|Z = \underline{0}) \leq \frac{r}{M} \log M \leq 2.
\end{equation}
Then by (\ref{eq:pairRandom}), (\ref{eq:Injective}), (\ref{eq:Subadditivity}) and \eqref{eq:prefix} we have
\begin{equation} \label{eq:UBonZ}
\sum_{i=1}^r |P_i|^2 \leq q^{f (2q-1)/q + o(n)}
\end{equation}
where $o(n)$ is meant as $n \to \infty$.

Finally we are now ready to prove Theorem \ref{theorem1}. By (\ref{eq2}) and (\ref{eq:UBonZ}) we have that
\begin{equation}\label{eq10}
M^2 \leq q^{f(2q-1)/q + e + o(n)}.
\end{equation}
Since $e$ is fixed and we know that $n=e+f$, from \eqref{eq10} we get
$$
	M \leq q^{n\frac{2q-1}{3q-1} + o(n)}
$$
and Theorem \ref{theorem1} follows.
\end{proof}

In Figure \ref{fig:Comparison} we give a comparison between the bounds on the rate of $\overline{2}$-separable codes given in \cite{Gu}, \cite{Wang} and the one given in Theorem \ref{theorem1}.

\section{Bounds for $B_2$ codes}
The related notion of $B_2$ codes can be introduced as follows.
\begin{Definition}
We say that $C_n = \{c_1, c_2, \ldots, c_M\}$ is a $q$-ary $B_2$ code with $M$ codewords of length $n$ and with symbols in the alphabet $[0, q-1]$ if all sums (over the real field) $c_i+c_j$ for $1\leq i \leq j \leq n$ are different.
\end{Definition}
Note that for $q=2$ this definition is equivalent to the definition of a $\bar{2}$-separable code.
Gu et al. in \cite{Gu} provide non trivial bounds on the rate of $q$-ary $B_2$ codes and they also observed that an implicit upper bound can be found in Lindtr\"om \cite[Theorem 1]{Lind}.
These bounds were improved by Wang \cite{Wang} for every ${2 \leq q \leq 12}$. 

An immediate extension of the method presented in the previous section leads to the following.
\begin{Theorem}\label{theorem2} For integer $q \geq 2$, let $C_n$ be a family of $q$-ary $B_2$ codes. Then
$$R_b :=  \limsup_{n \to \infty} \frac{1}{n} \log_q |C_n| \leq \frac{q+(q-1)\log_q2}{2q+(q-1)\log_q2}.$$
\end{Theorem}
For every ${3 \leq q \leq 12}$ it improves the one given in \cite{Gu} but not the bound given in \cite{Wang}. Of course, when ${q=2}$, Theorems \ref{theorem1} and \ref{theorem2} give the same bound consistently with the fact that any binary $\overline{2}$-separable code is also a $B_2$ code and vice versa.

\begin{proof}[Proof of Theorem \ref{theorem2}]
In this case, we consider the set $D = \{-q+1, \ldots, -1,$ $0, 1, \ldots, q-1\}$, so that now $|D| = 2q - 1$.
For an integer $f$, we define the function $\Phi$ as $\Phi(w_1, w_2) = w_1 - w_2 \in D^f$
where $w_1, w_2 \in [0, q-1]^f$ and the difference is computed in $\mathbb{Z}$.

Let $C_n$ be a $q$-ary $B_2$ code and suppose we have constructed the $r=q^e$ subcodes $P_i$ (as done in Theorem \ref{theorem1}). It can be proved, in a similar manner as Theorem \ref{theorem1}, that $\Phi$ is injective when we restrict its domain to pairs of suffixes of different vectors taken from the same $P_i$ for all $i=1, \ldots, r$. Then, the procedure used in Theorem \ref{theorem1} can be applied to prove Theorem \ref{theorem2}. All the equations from \eqref{eq2} to \eqref{eq10} are verified, also in this case, with the only difference that the cardinality of the set $D$ is $2q - 1$ and not $q^2-q+1$. 
\end{proof}

In Figure \ref{fig:Comparison2} we give a comparison between the bounds on the rate of $B_2$ codes given in \cite{Gu}, \cite[Theorem 1]{Lind}, \cite{Wang} and the one given in Theorem \ref{theorem2}.
\begin{figure}[h]
  \centering
  \begin{minipage}[b]{0.495\textwidth}
    \includegraphics[width=\textwidth]{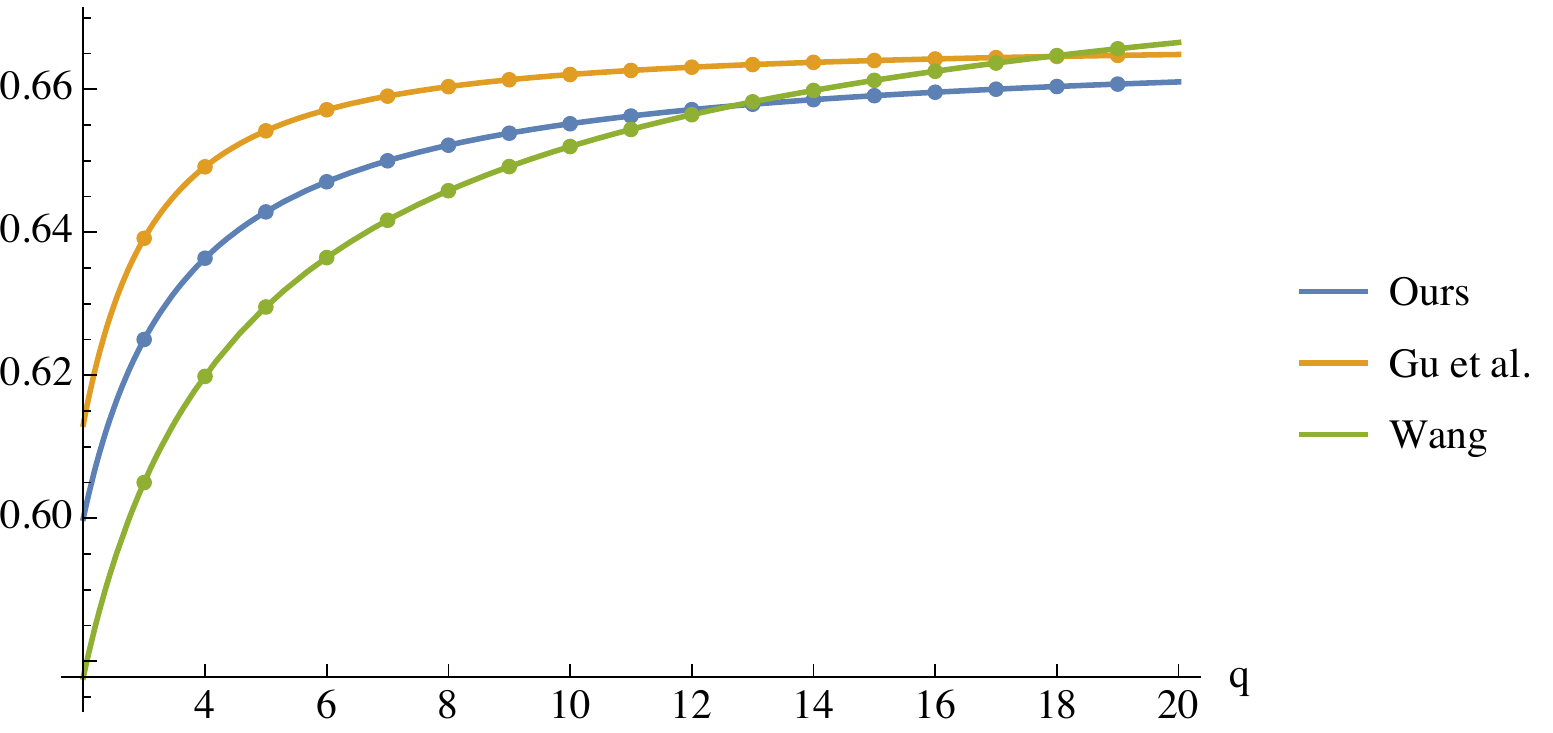}
    \caption{Upper bounds on $R_s$.}
    \label{fig:Comparison}
  \end{minipage}
  \begin{minipage}[b]{0.495\textwidth}
    \includegraphics[width=\textwidth]{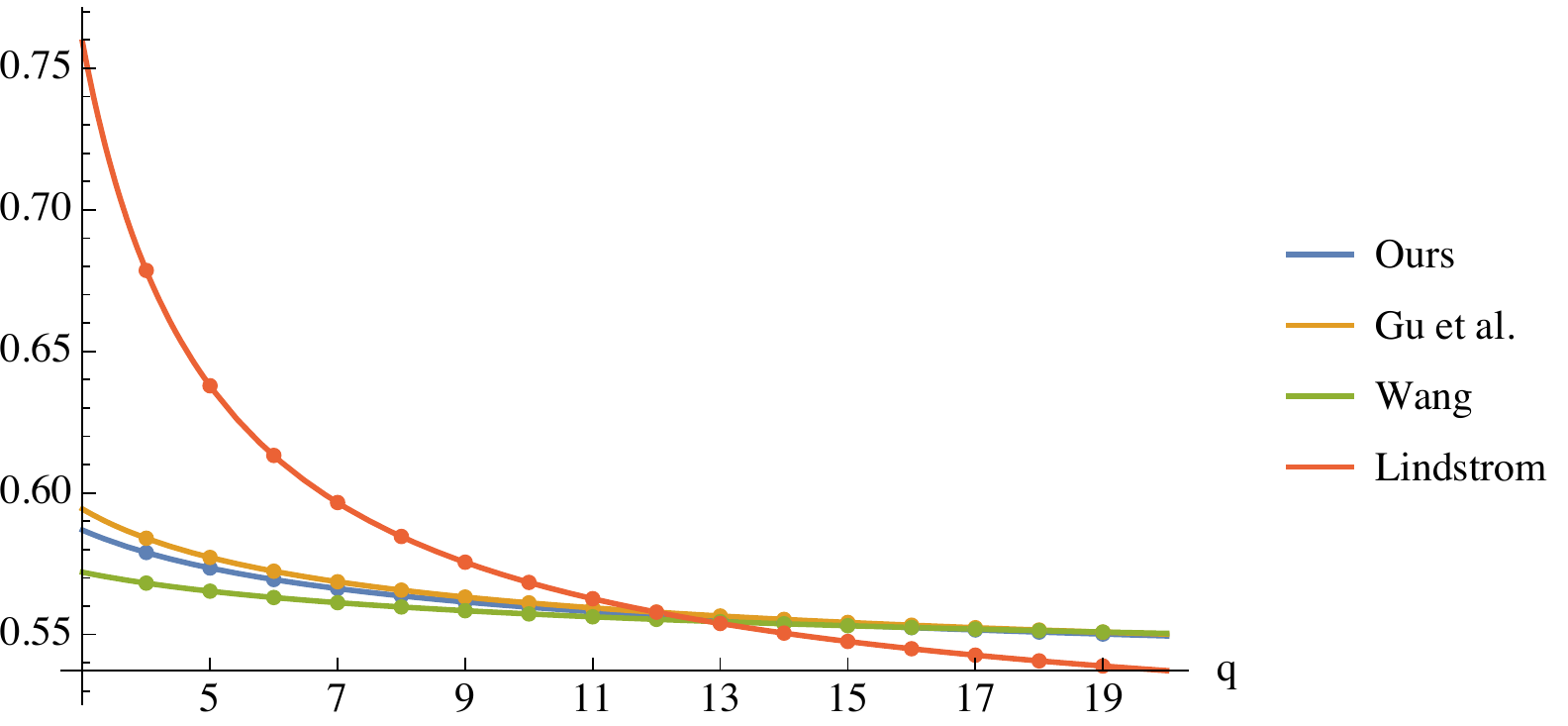}
    \caption{Upper bounds on $R_b$.}
    \label{fig:Comparison2}
  \end{minipage}
\end{figure}


%
%


\begin{thebibliography}{12}
\bibitem{Lind} B. Lindstr\"om \emph{On $B_2$-Sequences of Vectors}, \textit{Journal Of Number Theory} 4, 261-265, 1972.
\bibitem{Gu} Y. Gu, J.Fan, and Y. Miao \emph{Improved Bounds for Separable Codes and $B_2$ Codes}, \textit{IEEE Communications Letters} vol. 24, no. 1, January 2020.
\bibitem{Wang} X. Wang \emph{Improved upper bounds for parent-identifying set systems and separable codes}, \textit{Designs, Codes and Cryptography}, October 2020.
\bibitem{Miao2} M. Cheng and Y. Miao, \emph{On anti-collusion codes and detection algo- rithms for multimedia fingerprinting,} IEEE Trans. Inf. Theory, vol. 57, no. 7, pp. 4843-4851, Jul. 2011.
\bibitem{Blackburn1} S. R. Blackburn, \emph{Frameproof codes,} SIAM J. Discrete Math., vol. 16, no. 3, pp. 499-510, 2003.
\bibitem{Blackburn2} S. R. Blackburn, \emph{Probabilistic existence results for separable codes,} IEEE Trans. Inf. Theory, vol. 61, no. 11, pp. 5822-5827, Nov. 2015.
\bibitem{Cohen} G. Cohen, S. Litsyn, and G. Z\'emor, \emph{Binary $B_2$ sequences: A new upper bound,} J. Combinat. Theory, Ser. A, vol. 94, no. 1, pp. 152-155, Apr. 2001.
\bibitem{CoverThomas} T. M. Cover and J. A. Thomas,  \emph{Elements of information theory}, John Wiley \& Sons, 2012.
\bibitem{Dyachkov} Arkadii D'yachkov,  Nikita Polyanskii, Vladislav Shchukin and Ilya Vorobyev, \emph{Separable Codes for the Symmetric Multiple-Access Channel}, IEEE Trans. Inf. Theory, vol. 65, no. 6, pp. 3738-3750, 2019.
\end{thebibliography}


\end{document}